\documentclass[draft,reqno]{amsproc}
\usepackage{amsmath,amssymb,amsthm,euscript,mathrsfs,mathtools,microtype,needspace,tikz}
\usepackage{refcheck}
\usepackage{caption}
\usetikzlibrary{arrows.meta}
\tikzset{
  ain vertex/.style={circle,fill=black!8,draw=none,
    minimum size=7mm,inner sep=1pt,font=\normalsize},
  ain edge/.style={-{Stealth[length=4.5pt,width=3.5pt]},
    line width=.4pt},
  ain continuation/.style={ain edge,dash pattern=on 2pt off 2pt}
}

\numberwithin{equation}{section}

\newtheorem{thm}{Theorem}[section]
\newtheorem{lem}[thm]{Lemma}
\newtheorem{prop}[thm]{Proposition}

\newtheorem*{problema}{{\bf The ``moment property implies subnormality''
problem}}
\newtheorem*{problemb}{{\bf The injectivity problem}}
\theoremstyle{remark}
\newtheorem{remark}[thm]{Remark}

\DeclareMathOperator{\supp}{supp}

\DeclareMathOperator{\D}{d\hspace{-0.25ex}}
\newcommand{\Borel}{\mathfrak{B}}
\newcommand{\CC}{\mathrm{(CC)}}
\newcommand{\Dinf}{\mathcal{D}^{\infty}}
\newcommand*{\dz}[1]{{\EuScript D}(#1)}
\newcommand*{\gammab}{{\boldsymbol{\gamma}}}
\newcommand*{\gcal}{{\mathscr G}}
\newcommand*{\Ge}{\geqslant}
\newcommand*{\hfi}{h_{\phi}}
\newcommand*{\Le}{\leqslant}
\newcommand{\N}{\mathbb{N}}

\newcommand{\Rplus}{\mathbb{R}_{+}}

\newcommand{\Zp}{\mathbb{Z}_{+}}

\title[``Moment property implies
subnormality'' and injectivity problems] {Solutions to the ``moment
property implies subnormality'' and injectivity problems for
composition operators on one-circuit graphs}

\author[P. Budzy\'nski]{Piotr Budzy\'nski}
\address{Katedra Zastosowa\'n Matematyki, Uniwersytet Rolniczy
w Krakowie,\newline ul.~Balicka 253c, PL-30198 Krak\'ow, Poland}
\email{piotr.budzynski@ur.krakow.pl}

\author[Z. J. Jab{\l}o\'nski]{Zenon Jan Jab{\l}o\'nski}
\address{Instytut Matematyki, Uniwersytet Jagiello\'nski,\newline
ul.~{\L}ojasiewicza 6, PL-30348 Krak\'ow, Poland}
\email{Zenon.Jablonski@im.uj.edu.pl}

\author[I. B. Jung]{Il Bong Jung}
\address{Department of Mathematics, Kyungpook National University,\newline
Daegu 41566, Republic of Korea}
\email{ibjung@knu.ac.kr}

\author[J. Stochel]{Jan Stochel}
\address{Instytut Matematyki, Uniwersytet Jagiello\'nski,\newline
ul.~{\L}ojasiewicza 6, PL-30348 Krak\'ow, Poland}
\email{Jan.Stochel@im.uj.edu.pl}
\date{}

\begin{document}

   \begin{abstract}
We consider composition operators in $L^2(\mu)$, where
$\mu$ is a discrete measure. Our main result shows
that, for every integer $p\Ge2$, there exists a
non-hyponormal composition operator that generates
Stieltjes moment sequences and whose symbol induces a
directed graph consisting of a circuit of length $p$
and one infinite branch. In contrast, if the symbol
induces a directed graph consisting of one loop and
one infinite branch, then the generation of Stieltjes
moment sequences does imply subnormality. These two
results completely solve the ``moment property implies
subnormality'' problem left open in [Adv.\ Math.\ 310
(2017), 484-556] for the entire family of one-branch
directed graphs with a circuit of arbitrary length.
Combining them with the examples in that paper shows
that, for directed graphs consisting of one loop and a
finite number of infinite branches, $2$ is the
smallest number of branches for which the moment
property does not imply subnormality. Finally, we
construct a noninjective composition operator that
generates Stieltjes moment sequences and whose symbol
induces a directed graph consisting of one loop, one
infinite branch, and one leaf. This operator gives a
negative answer to the injectivity problem posed in
the aforementioned paper. Both constructions use the
Krein and Friedrichs measures of an arbitrary
normalized S-indeterminate Stieltjes moment sequence.
   \end{abstract}

\maketitle

   \section{Preliminaries}
   \subsection{Introduction}
We study when the property of generating Stieltjes
moment sequences, which we call the \emph{moment
property}, implies subnormality for composition
operators in $L^2$-spaces. Recall that an operator $A$
in a complex Hilbert space $\mathcal{H}$ is said to
\emph{generate Stieltjes moment sequences} if
$\Dinf(A):=\bigcap_{n=0}^{\infty}\dz{A^n}$ is dense in
$\mathcal{H}$ and $\{\|A^nf\|^2\}_{n=0}^{\infty}$ is a
Stieltjes moment sequence for every $f\in\Dinf(A)$,
where $\dz{A^n}$ denotes the domain of $A^n$ and
$A^0=I$ (see \cite[Section~1.2]{BJJS17}). For bounded
operators, the implication under consideration follows
from Lambert's characterization \cite{Lam76}; see also
\cite[Theorem~7]{SS89}. For unbounded composition
operators, however, the moment property need not imply
even hyponormality. The counterexamples constructed in
\cite{JJS12} involve a directed tree with an
infinitely branching vertex. This led to the questions
of whether such counterexamples could occur on locally
finite (connected) directed graphs and, more
generally, how simple their underlying graphs could
be.

In \cite{BJJS17}, simplicity was considered in terms
of local valency, understood as the number of outgoing
edges at a vertex (see \cite[Section~1.1 and
Remark~3.2.2]{BJJS17}). Theorem~5.5.2 of that paper
provides examples of non-hyponormal composition
operators generating Stieltjes moment sequences on
$\gcal_{2,0}$, the directed graph consisting of one
loop and two infinite branches. Its branching vertex
has valency three, while all other vertices have
valency one. The case of one loop and a single
infinite branch was explicitly left open in
\cite[Section~1.1]{BJJS17}. To motivate the study of
this case, we first discuss graphs without branching.

In studying the subnormality problem, we focus on
directed graphs that admit subnormal composition
operators (with symbols inducing these graphs), so
that failure of subnormality is not forced by the
graph itself. Such graphs have no vertices of valency
zero (see \eqref{suxj} and Proposition~\ref{wkwh}; see
also Theorem~\ref{noninjective}). If every vertex has
valency one, the corresponding composition operator is
unitarily equivalent to a bilateral weighted shift or
a finite-dimensional cyclic weighted shift
\cite[Remark~3.2.2]{BJJS17}. In both cases, the moment
property implies subnormality \cite[Theorems~5
and~7]{SS89}. Thus, within this class of directed
graphs, a counterexample requires at least one
branching vertex.

We therefore investigate the implication within the
family of one-branch directed graphs with circuits of
varying length. More precisely, for an integer
$\kappa\Ge0$, let $\gcal_{1,\kappa}$ denote the
directed graph consisting of a circuit of length
$p:=\kappa+1$ and one infinite branch attached to a
circuit vertex (see Figures~\ref{fig:longcircuit}
and~\ref{fig:onebranch}). Each of these graphs has
exactly one vertex of valency two, with all other
vertices having valency one. Thus the circuit length
varies without changing the local valencies. We
consider a composition operator $C_{\phi_{1,\kappa}}$
in the Hilbert space $L^2(\mu)$, where $\mu$ is a
discrete measure on the vertex set $X_{1,\kappa}$ of
$\gcal_{1,\kappa}$. Its symbol is the self-map
$\phi_{1,\kappa}\colon X_{1,\kappa} \to X_{1,\kappa}$,
which acts as a one-step shift: cyclically on the
circuit of length $p$ and towards the circuit along
the infinite branch. The precise definitions are given
in Subsection~\ref{subsec:graphs}. This leads to our
first problem.
   \begin{problema}
For which integers $\kappa\Ge 0$ does the generation
of Stieltjes moment sequences imply subnormality for
composition operators $C_{\phi_{1,\kappa}}$ on
$\gcal_{1,\kappa}${\em ?}
   \end{problema}
A second question concerns a weaker consequence of
subnormality. Every subnormal composition operator is
injective \cite[Corollary~6.3]{BJJS14}, so it is
natural to ask whether the moment property alone
forces injectivity. This is the subject of the
following problem, posed in \cite{BJJS17} for general
$\sigma$-finite measure spaces.
   \begin{problemb}[{\cite[Problem~3.3.6]{BJJS17}}]
Suppose that $(X,\mathcal{A},\mu)$ is a $\sigma$-finite measure
space, $\phi\colon X \to X$ is a nonsingular self-map, and the
composition operator $C_\phi$ in $L^2(\mu)$ generates Stieltjes
moment sequences. Must $C_\phi$ be injective?
   \end{problemb}
We refer to Subsection~\ref{subsec:moments} for the moment-theoretic
terminology and to Subsection~\ref{subsec:discrete} for further
operator-theoretic definitions and the notation $h_{\phi^n}$ for the
iterated Radon-Nikodym derivatives.

Our main result settles the first problem in the
negative for every $\kappa\Ge1$.
   \begin{thm} \label{cnyz}
For every integer \(\kappa\Ge 1\), there exists a
discrete measure \(\mu\) on \(X_{1,\kappa}\) such that
the composition operator \(C_\phi\) in \(L^2(\mu)\)
with \(\phi=\phi_{1,\kappa}\) has the
following~properties{\em :}
   \begin{enumerate}
   \item[\textup{(i)}]
\(C_\phi\) generates Stieltjes moment sequences,
   \item[\textup{(ii)}]
\(C_\phi\) is not hyponormal and, consequently, is not
subnormal.
   \end{enumerate}
   \end{thm}
Each of the underlying graphs in Theorem~\ref{cnyz}
has exactly one branching vertex of valency two, while
all other vertices have valency one. By Lambert's
characterization of subnormality, the corresponding
composition operators are necessarily unbounded. They
are also injective, since their symbols are surjective
(see \eqref{suxj}). In particular, neither infinite
branching nor failure of injectivity is needed for the
moment property to coexist with non-hyponormality. The
proof of Theorem~\ref{cnyz} is given in
Subsection~\ref{sec:long}. The underlying directed
graph is shown in Figure~\ref{fig:longcircuit}.

   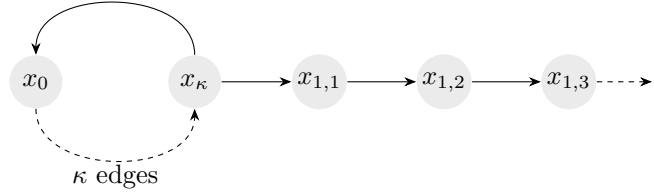
\begin{figure}[htbp]
   \centering
   \begin{tikzpicture}[x=1cm,y=1cm]
\node[ain vertex] (b0) at (-2.10,0) {$x_0$}; \node[ain vertex] (bk)
at (0,0) {$x_\kappa$}; \node[ain vertex] (b1) at (1.65,0)
{$x_{1,1}$}; \node[ain vertex] (b2) at (3.30,0) {$x_{1,2}$};
\node[ain vertex] (b3) at (4.95,0) {$x_{1,3}$};

\draw[ain edge] (bk.north) .. controls (0,1.27) and (-2.10,1.27) ..
(b0.north);

\draw[ain continuation] (b0.south) .. controls (-2.10,-1.27) and
(0,-1.27) .. (bk.south); \node at (-1.05,-1.25) {$\kappa$ edges};

\draw[ain edge] (bk) -- (b1); \draw[ain edge] (b1) -- (b2); \draw[ain
edge] (b2) -- (b3); \draw[ain continuation] (b3) -- (6.05,0);
   \end{tikzpicture}
\caption{\label{fig:longcircuit}The directed graph
$\gcal_{1,\kappa}$ in Theorem~\ref{cnyz}, where
$\kappa\Ge1$, consists of a circuit of length
$\kappa+1$ and one infinite branch. The arrows
$(\phi(y),y)$ point in the direction opposite to the
action of~$\phi$.}
   \end{figure}

When the circuit is a loop, the conclusion is
different. In fact, subnormality can be characterized
by the moment property at only two vertices: the loop
vertex $x_0$ and the first branch vertex $x_{1,1}$.
   \begin{thm} \label{kezo}
Let $\mu$ be a discrete measure on $X_{1,0}$ and let
$\phi=\phi_{1,0}$. Then the following conditions are
equivalent for the composition operator $C_\phi$ in
$L^2(\mu)${\em :}
   \begin{enumerate}
   \item[\textup{(i)}]
$C_\phi$ is subnormal,
   \item[\textup{(ii)}]
$C_\phi$ generates Stieltjes moment sequences,
   \item[\textup{(iii)}]
$\{h_{\phi^n}(x_0)\}_{n=0}^{\infty}$ and
$\{h_{\phi^n}(x_{1,1})\}_{n=0}^{\infty}$ are Stieltjes
moment sequences.
   \end{enumerate}
   \end{thm}
No determinacy assumption is imposed on either moment
sequence in Theorem~\ref{kezo} (cf.\
\cite[Theorem~41]{BJJS15}). This theorem settles the
case $\gcal_{1,0}$ left open in \cite{BJJS17} and
establishes the optimality of the lower bound of two
branches in \cite[Theorem~5.5.2]{BJJS17} within the
one-loop family. The proof of Theorem~\ref{kezo} is
given in Subsection~\ref{sec:loop}. The corresponding
directed graph is shown in Figure~\ref{fig:onebranch}.

   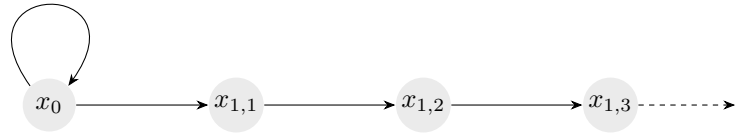
\begin{figure}[htbp]
   \centering
\begin{tikzpicture}[x=1.5cm,y=1.5cm]
\node[ain vertex] (a0) at (0,0) {$x_0$}; \node[ain
vertex] (a1) at (1.65,0) {$x_{1,1}$}; \node[ain
vertex] (a2) at (3.30,0) {$x_{1,2}$}; \node[ain
vertex] (a3) at (4.95,0) {$x_{1,3}$};

\draw[ain edge] (a0.north west) .. controls
(-.85,1.12) and (.85,1.12) .. (a0.north east);
\draw[ain edge] (a0) -- (a1); \draw[ain edge] (a1) --
(a2); \draw[ain edge] (a2) -- (a3); \draw[ain
continuation] (a3) -- (6.05,0);
\end{tikzpicture}
\caption{\label{fig:onebranch}The directed graph
$\gcal_{1,0}$ in Theorem~\ref{kezo} consists of one
loop and one infinite branch. The arrows $(\phi(y),y)$
point in the direction opposite to the action
of~$\phi$.}
   \end{figure}
Together, Theorems~\ref{cnyz} and~\ref{kezo}
completely solve the first problem. Although the
directed graphs $\gcal_{1,\kappa}$, $\kappa \Ge 0$,
have the same local valencies, the answer changes when
the loop is replaced by a longer circuit.

For longer circuits, additional moment-theoretic
assumptions can restore subnormality.
Proposition~\ref{cpws} shows that, if the sequences of
iterated Radon-Nikodym derivatives at $x_0$ and
$x_{1,1}$ are Stieltjes moment sequences, then
S-determinacy of the subsequence at $x_0$ indexed by
multiples of the circuit length $p$ implies
subnormality. This result and the corresponding
sufficient conditions involving the Carleman condition
are discussed in Subsection~\ref{sec:det}.

Our next result gives a negative solution to the
injectivity problem.
   \begin{thm} \label{noninjective}
There exist a discrete measure \(\mu\) on a set \(X\)
and a self-map $\phi\colon X\to X$ such that the
composition operator \(C_\phi\) in \(L^2(\mu)\)
generates Stieltjes moment sequences but is not
injective.
   \end{thm}
The composition operator $C_{\phi}$ constructed in the
proof of Theorem~\ref{noninjective} (see
Section~\ref{sec:inj}) acts in an $L^2$-space over a
discrete measure space, and its symbol $\phi$ has
finite fibers. The directed graph $\gcal$ induced by
$\phi$ is obtained from $\gcal_{1,0}$ by attaching a
single leaf $z$ to the loop vertex $x_0$ (see
Figure~\ref{fig:leaf}). The kernel of $C_{\phi}$ is
one-dimensional; more precisely, $\ker
C_\phi=\mathbb{C}\cdot\chi_{\{z\}}$. The operator
$C_{\phi}$ is unbounded and not hyponormal, whereas
removing the leaf and retaining all the remaining
point masses yields a subnormal composition operator
on $\gcal_{1,0}$. These features are discussed in
Remarks~\ref{unbco} and~\ref{gstr} following the proof
of Theorem~\ref{noninjective}. Thus the injectivity
problem posed in \cite[Problem~3.3.6]{BJJS17} is
completely settled in the negative by a counterexample
on a locally finite directed graph with only one
branching vertex.

The constructions in the proofs of Theorems~\ref{cnyz}
and~\ref{noninjective} start from the Krein and
Friedrichs measures of an arbitrary S-indeterminate
Stieltjes moment sequence whose zeroth term is one.
The relevant properties of these measures, including
their behavior under affine transformations, are
recalled in Subsection~\ref{subsec:extremal}. These
measures have identical moments but different
supports, and this difference provides the flexibility
needed in the constructions. Neither construction
requires explicit formulas for the initial
representing measures. For explicit examples and
related results, see
\cite{AlCar65,B25,BS26,BV94,Chi68,Chr03,IsMa94,Pe95};
see also \cite[(5.5.9), (5.5.10), and
Remark~5.5.4]{BJJS17}.

The proof of Theorem~\ref{cnyz} relies on the $p$-step
difference relation \eqref{circuit-diff} for the
moment sequence $\{h_{\phi^n}(x_0)\}_{n=0}^{\infty}$
at the circuit vertex $x_0$, where $p=\kappa+1$ is the
circuit length. After applying suitable affine
transformations to the initial representing measures,
we construct a Stieltjes moment sequence whose
$p$-step differences also form a Stieltjes moment
sequence. The resulting composition operator
nevertheless fails the hyponormality criterion. For a
loop, by contrast, the first-difference relation,
together with the moment assumptions in
Theorem~\ref{kezo}, implies subnormality.

In the proof of Theorem~\ref{noninjective}, the atom
of the Krein measure at zero plays a different role.
Attaching the leaf changes all moments at the loop
vertex except the moment with index zero. We realize
this change in the representing measure by
transferring part of the mass concentrated at $0$ to
the point $1$, without changing the total mass (see
\eqref{pzoo}). The Friedrichs measure provides
positive representing measures at the vertices of the
infinite branch. The two constructions thus separate
the moment property from subnormality and injectivity,
respectively.

   \subsection{Moment sequences}\label{subsec:moments} In what follows,
$\mathbb{R}$ and $\mathbb{C}$ stand for the fields of real and
complex numbers, respectively. We write $\N=\{1,2,\ldots\}$,
$\Zp=\{0,1,2,\ldots\}$, and $\Rplus=[0,\infty)$. The symbol
$\Borel(Y)$ denotes the Borel $\sigma$-algebra of a topological space
$Y$, and $\chi_E$ denotes the characteristic function of a subset $E$
of a set $X$. Unless stated otherwise, all measures are positive, and
$\supp(\nu)$ denotes the closed support of a finite Borel measure
$\nu$ on $\mathbb{R}$. The symbol $\delta_t$ is reserved for the
Borel probability measure on $\mathbb{R}$ with
$\supp(\delta_t)=\{t\}$, where $t\in\mathbb{R}$. We use the
convention $0^0=1$ when evaluating the monomial $x^0$ at $x=0$.

A sequence $\gammab=\{\gamma_n\}_{n=0}^{\infty}$ of real numbers is
called a {\em Stieltjes moment sequence} if there exists a Borel
measure $\nu$ on $\Rplus$ such that
   \begin{equation} \label{dbsc}
\gamma_n=\int_{\Rplus}t^n\D \nu(t), \quad n\in\Zp.
   \end{equation}
Such a measure is called an {\em S-representing measure} of
$\gammab$. If it is unique, $\gammab$ is said to be {\em
S-determinate}; otherwise, it is {\em S-indeterminate}. Allowing
representing measures on $\mathbb{R}$ in \eqref{dbsc} leads to the
corresponding notions of a {\em Hamburger moment sequence}, an {\em
H-representing measure}, and {\em H-determinacy}. All moments
appearing in these definitions are finite. We refer to
\cite{BCR84,Simon98,Sch17} for general background.

We shall use the following form of the support criterion in
\cite[Lemma~2.1.6]{BJJS17}. A sequence
$\gammab=\{\gamma_n\}_{n=0}^{\infty}\subseteq\Rplus$ is a Stieltjes
moment sequence admitting an S-representing measure supported in
$[1,\infty)$ if and only if
   \begin{equation}\label{support-hankel}
0\Le\sum_{i,j=0}^{n}\gamma_{i+j}\lambda_i\overline{\lambda_j}
\Le\sum_{i,j=0}^{n}\gamma_{i+j+1}\lambda_i\overline{\lambda_j}
   \end{equation}
for all $n\in\Zp$ and $\lambda_0,\ldots,\lambda_n\in\mathbb{C}$.
Equivalently, $\gammab$ is a Stieltjes moment sequence and its
first-difference sequence $\{\gamma_{n+1}-\gamma_n\}_{n=0}^{\infty}$
is a Hamburger moment sequence.

For later reference, we also state the parts of
\cite[Lemma~4.3.1]{BJJS17} concerning subsequences and differences.
If $\gammab=\{\gamma_n\}_{n=0}^{\infty}$ is a Stieltjes moment
sequence and $p\in\N$, then $\{\gamma_{jp}\}_{j=0}^{\infty}$ is a
Stieltjes moment sequence. If this subsequence is S-determinate, then
so is $\gammab$. If, in addition,
$\{\gamma_{(j+1)p}-\gamma_{jp}\}_{j=0}^{\infty}$ is a Stieltjes
moment sequence, then the unique S-representing measure of $\gammab$
is supported in $[1,\infty)$.

The {\em Carleman condition} for a Stieltjes moment sequence
$\gammab=\{\gamma_n\}_{n=0}^{\infty}$ is
   \begin{equation*}
\sum_{n=1}^{\infty}\gamma_n^{-1/(2n)}=\infty,
   \end{equation*}
where a term corresponding to $\gamma_n=0$ is interpreted as
$\infty$. This condition implies H-determinacy and hence
S-determinacy (see \cite[Proposition~2.5.1(i)]{BJJS17}). Moreover, it
follows from \cite[Lemma~4.3.1(iii)]{BJJS17} that if both $\gammab$
and $\{\gamma_{n+1}-\gamma_n\}_{n=0}^{\infty}$ are Stieltjes moment
sequences, then
   \begin{align} \label{carleman-diff}
   \begin{minipage}{67ex}
$\gammab$ satisfies the Carleman condition if and only if
$\{\gamma_{n+1}-\gamma_n\}_{n=0}^{\infty}$ satisfies the Carleman
condition.
   \end{minipage}
   \end{align}
   \subsection{Moment property for composition operators}
\label{subsec:discrete} Let $A$ be an operator in a complex Hilbert
space $\mathcal{H}$. We denote its domain by $\dz{A}$ and set
$\Dinf(A)=\bigcap_{n=0}^{\infty}\dz{A^n}$, where $A^0=I$. As in
\cite[Section~1.2]{BJJS17}, we say that $A$ {\em generates Stieltjes
moment sequences} if $\Dinf(A)$ is dense in $\mathcal{H}$ and
$\{\|A^nf\|^2\}_{n=0}^{\infty}$ is a Stieltjes moment sequence for
every $f\in\Dinf(A)$. An operator is {\em subnormal} if it is densely
defined and has a normal extension in a possibly larger Hilbert
space. It is {\em hyponormal} if it is densely defined,
$\dz{A}\subseteq\dz{A^*}$, and $\|A^*f\|\Le\|Af\|$ for every
$f\in\dz{A}$. Subnormal operators are hyponormal. For the
fundamentals of the theory of subnormal operators, we refer to
\cite{Con91}.

By a {\em discrete measure} on a nonempty set $X$ we mean a
$\sigma$-finite measure $\mu$ defined on the $\sigma$-algebra $2^X$
such that $\mu(x)>0$ for every $x\in X$, where
   \begin{align*}
\mu(x):=\mu(\{x\}), \quad x\in X.
   \end{align*}
Then $X$ is at most countable and $\mu(x)<\infty$ for every $x\in X$.
Thus, on a nonempty at most countable set $X$, discrete measures are
identified with strictly positive finite-valued functions on $X$.

Let $\mu$ be a discrete measure on $X$ and let $\phi\colon X \to X$
be a self-map. The {\em composition operator} $C_\phi$ in $L^2(\mu)$
with {\em symbol} $\phi$ is defined by
\begin{align*}
\dz{C_\phi}&=\{f\in L^2(\mu):f\circ\phi\in L^2(\mu)\},\\
C_\phi f&=f\circ\phi, \quad f\in\dz{C_\phi}.
\end{align*}
Every self-map $\phi\colon X\to X$ is nonsingular, that is,
$\mu\circ\phi^{-1}$ is absolutely continuous with respect to $\mu$.
Here
   \begin{align*}
(\mu\circ\phi^{-1})(\varDelta)=\mu(\phi^{-1}(\varDelta)), \quad
\varDelta\subseteq X.
   \end{align*}
The Radon-Nikodym derivative $\hfi=\D \, (\mu\circ\phi^{-1})/\D\mu$,
with values in $[0,\infty]$, is given~by
   \begin{align}\label{hfi}
\hfi(x)=\frac{\mu(\phi^{-1}(\{x\}))}{\mu(x)}, \quad x\in X.
   \end{align}
We call $\phi^{-1}(\{x\})$ the {\em fiber} of $\phi$ over \(x\).
Applying \eqref{hfi} to the $n$th iterate $\phi^n$, with $\phi^0$ the
identity map on $X$, gives
\begin{align}\label{hfin}
h_{\phi^n}(x)=\frac{\mu(\phi^{-n}(\{x\}))}{\mu(x)},
\quad x\in X,\ n\in\Zp.
\end{align}
As usual, $\phi^{-n}(\varDelta)$ means $(\phi^n)^{-1}(\varDelta)$
whenever $\varDelta \subseteq X$ and $n \in \Zp$. As in
\cite[(3.1.4)]{BJJS17}, we have
   \begin{equation}\label{pointwise-norm}
\text{$\|C_\phi^n\chi_{\{x\}}\|^2=\mu(x)h_{\phi^n}(x)$ whenever $n\in
\Zp$ and $\chi_{\{x\}}\in\dz{C_\phi^n}$.}
   \end{equation}
Observe that, by \eqref{hfi},
   \begin{align} \label{suxj}
   \begin{minipage}{65ex}
$C_\phi$ is injective $\iff$ $\phi(X)=X$ $\iff$ $h_\phi(x)>0$ for
every $x\in X$.
   \end{minipage}
   \end{align}

   Finally, as shown below, a composition operators $C_{\phi}$ over
discrete measure space, whose symbol $\phi$ has finite fibers,
generates Stieltjes moment sequences if and only if the iterated
Radon-Nikodym derivatives $\{h_{\phi^n}(x)\}_{n=0}^{\infty}$ form a
Stieltjes moment sequence for every $x\in X$.
   \begin{lem} \label{wkws}
Let $\mu$ be a discrete measure on $X$ and let $\phi\colon X \to X$
be a self-map such that $\phi^{-1}(\{x\})$ is finite for every $x\in
X$. Then
   \begin{enumerate}
   \item[(i)] $\overline{\Dinf(C_\phi)}=L^2(\mu)$,
   \item[(ii)] $C_{\phi}$ generates Stieltjes moment sequences if and
only if $\{h_{\phi^n}(x)\}_{n=0}^{\infty}$ is a Stieltjes moment
sequence for every $x\in X$.
   \end{enumerate}
   \end{lem}
   \begin{proof}
(i) If $\phi^{-1}(\{x\})$ is finite for every $x\in X$, then the
iterated fiber $\phi^{-n}(\{x\})$ is finite for every $x\in X$ and
every $n\in \N$. The space of all finitely supported complex
functions on $X$ is invariant under $C_\phi$. Since it is dense in
$L^2(\mu)$, we conclude that $\overline{\Dinf(C_\phi)}=L^2(\mu)$.
This also follows from \cite[(3.1.5)]{BJJS17} and
\cite[Theorem~47]{BJJS18}.

   (ii) By assumption on fibers of $\phi$ and \eqref{hfin},
$h_{\phi^n}(x)<\infty$ for all $x\in$ and $n\in \N$. If $C_{\phi}$
generates Stieltjes moment sequences, then, in view of the previous
argument and \eqref{pointwise-norm}, the sequence
$\{h_{\phi^n}(x)\}_{n=0}^{\infty}$ is a Stieltjes moment sequence for
every $x\in X$.

Assume now that for every $x\in X$, the sequence
$\{h_{\phi^n}(x)\}_{n=0}^{\infty}$ is a Stieltjes moment sequence
with a representing measure $P_x$. Take \(f\in\Dinf(C_\phi)\) and
define a finite positive Borel measure \(\nu_f\) on \(\mathbb{R}_+\)
by
   \begin{align*}
\nu_f(\varDelta) = \sum_{x\in X}\mu(x)|f(x)|^2 P_x(\varDelta), \quad
\varDelta\in\Borel(\mathbb{R}_+).
   \end{align*}
Since every \(P_x\) is a probability measure,
   \begin{align*}
\nu_f(\mathbb{R}_+) = \sum_{x\in X}\mu(x)|f(x)|^2 = \|f\|^2<\infty.
   \end{align*}
Using Tonelli's theorem, we obtain \allowdisplaybreaks
   \allowdisplaybreaks
   \begin{align*}
\int_{\Rplus}t^n\D \nu_f(t) &= \sum_{x\in X} \mu(x)|f(x)|^2
\int_{\Rplus}t^n \D P_x(t)
   \\
&= \sum_{x\in X} \mu(x)|f(x)|^2h_{\phi^n}(x)
   \\
&\hspace{-1ex} \overset{\eqref{hfin}}{=} \sum_{x\in X} |f(x)|^2
\mu\bigl(\phi^{-n}(\{x\})\bigr)
   \\
&= \sum_{x\in X} \sum_{y\in X} |f(x)|^2 \chi_{\phi^{-n}(\{x\})}(y)
\mu(y)
   \\
&= \sum_{y\in X} \mu(y) \sum_{x\in X} |f(x)|^2 \bigl(\chi_{\{x\}}
\circ \phi^{n}\bigr)(y)
   \\
&= \sum_{y\in X} \mu(y)|f(\phi^n(y))|^2 = \|C_\phi^n f\|^2, \qquad n
\in \mathbb{Z}_+.
   \end{align*}
Together with (i), this shows that \(C_\phi\) generates Stieltjes
moment sequences.
   \end{proof}
   \subsection{One-circuit directed graphs}
\label{subsec:graphs} For $\eta\in\N\cup\{\infty\}$, put
   \begin{align*}
J_\eta=\{i\in\N:i\Le\eta\}.
   \end{align*}
Following \cite[(3.2.2) and (3.2.3)]{BJJS17}, for $\kappa\in\Zp$ let
   \begin{equation}\label{graph-set}
X_{\eta,\kappa}=\{x_0,\ldots,x_\kappa\}
\cup\{x_{i,j}:i\in J_\eta,\ j\in\N\},
   \end{equation}
where all the displayed points are distinct, and define
   \begin{equation}\label{graph-map}
\phi_{\eta,\kappa}(x)=
\begin{cases}
x_{i,j-1}&\text{if }x=x_{i,j},\ i\in J_\eta,\ j\Ge2,\\
x_\kappa&\text{if }x=x_0\text{ or }x=x_{i,1},\ i\in J_\eta,\\
x_{r-1}&\text{if }x=x_r,\ r \in J_{\kappa}.
\end{cases}
   \end{equation}
Given a set $X$, the directed graph {\em induced} by a self-map
$\phi\colon X\to X$ is understood to be the directed graph
$(X,E^\phi)$, where $X$ is regarded as the set of vertices and
$E^\phi:=\{(x,y)\in X\times X\colon x=\phi(y)\}$ as the set of edges.
Thus its arrows run in the direction opposite to the action of
$\phi$. We write $\gcal_{\eta,\kappa}$ for the directed graph induced
by $\phi_{\eta,\kappa}$ (see \eqref{graph-map}. It has $\eta$
infinite branches issuing from $x_\kappa$ and a circuit of {\em
length}
   \begin{align*}
p:=\kappa+1.
   \end{align*}
The circuit is a loop when $\kappa=0$. Every vertex has a nonempty
preimage, so all composition operators with this symbol are injective
(see \eqref{suxj}). If $\eta<\infty$, then every fiber is finite, so
Lemma~\ref{wkws} applies to every discrete measure on
$X_{\eta,\kappa}$.

Whenever we consider $C_\phi$ on $\gcal_{\eta,\kappa}$, we mean that
$X=X_{\eta,\kappa}$, $\phi=\phi_{\eta,\kappa}$, and $\mu$ is a
discrete measure on $X$. This is exactly the setting of
\cite[(3.4.1)]{BJJS17}. The identities for the iterated Radon-Nikodym
derivatives needed in this paper are
   \begin{align} \label{cycle-id}
h_{\phi^{n+r}}(x_0) &=\frac{\mu(x_r)}{\mu(x_0)}h_{\phi^n}(x_r),
&&n\in\Zp,\ 0\Le r\Le\kappa,
   \\  \label{circuit-diff}
h_{\phi^{n+p}}(x_0) &=h_{\phi^n}(x_0) + \sum_{i\in
J_\eta}\frac{\mu(x_{i,1})}{\mu(x_0)} h_{\phi^n}(x_{i,1}), &&n\in\Zp,
   \\ \label{branch-id}
h_{\phi^n}(x_{i,j}) &=\frac{\mu(x_{i,j+n})}{\mu(x_{i,j})},
&&n\in\Zp,\ i\in J_\eta,\ j\in\N.
   \end{align}
Formula \eqref{cycle-id} is \cite[(3.4.8)]{BJJS17}, while
\eqref{circuit-diff} is \cite[(3.4.9)]{BJJS17}. In turn, identity
\eqref{branch-id} follows directly from \eqref{hfin}.
   \subsection{Hyponormality and the consistency condition}
   Using \eqref{suxj}, we may restate the hyponormality criterion for
composition operators given in \cite{BJJS17} as follows.
   \begin{prop}[{\cite[Proposition~3.1.2]{BJJS17}}] \label{wkwh}
Let $\mu$ be a discrete measure on a set $X$ and let $\phi\colon X
\to X$ be a self-map. Assume that $C_{\phi}$ is densely defined. Then
$C_{\phi}$ is hyponormal if and only if $\phi(X)=X$ and
   \begin{equation} \label{hypcriterion}
\frac{1}{\mu(x)}\sum_{y\in\phi^{-1}(\{x\})}
\frac{\mu(y)^2}{\mu(\phi^{-1}(\{y\}))}\Le1, \quad x\in X.
   \end{equation}
Under these assumptions the denominators are positive and finite.
   \end{prop}
Suppose that $\mu$ is a discrete measure on a set $X$, and
$\phi\colon X \to X$ is a self-map. We say that a family
$\boldsymbol{P}=\{P_x\}_{x\in X}$ of Borel probability measures on
$\Rplus$ satisfies the {\em consistency condition} if
   \begin{equation}\label{consistency}
\frac{1}{\mu(x)}\sum_{y\in\phi^{-1}(\{x\})}\mu(y)P_y(\varDelta)
=\int_\varDelta t \D P_x(t), \quad x\in\phi(X),\
\varDelta\in\Borel(\Rplus).
   \end{equation}
This is condition $\CC$ of \cite[p.\ 505]{BJJS17}. The sufficient
condition for subnormality in \cite[Theorem~3.1.3]{BJJS17},
originating in \cite{BJJS15}, states that if $C_\phi$ is densely
defined, $\phi(X)=X$, and there exist a family $\boldsymbol{P}$ of
Borel probability measures on $\Rplus$ satisfying
\eqref{consistency}, then $C_\phi$ is subnormal. Moreover,
   \begin{align*}
h_{\phi^n}(x)=\int_{\Rplus} t^n \D P_x(t), \quad x\in X, \, n\in\Zp.
   \end{align*}
The surjectivity assumption is a part of this criterion.

For one-branch directed graph $\gcal_{1,\kappa}$, the criterion has a
particularly useful form.
   \begin{prop}[{\cite[Proposition~4.3.4]{BJJS17}}] \label{dirze}
Let $\kappa \in \Zp$, let $\mu$ be a discrete measure on $X_{1,
\kappa}$, and let $\phi=\phi_{1,\kappa}$. Then the following
conditions are equivalent{\em :}
   \begin{enumerate}
   \item[(i)] there exists a family $\{P_x\}_{x\in X}$ of Borel probability measures on $\Rplus$ satisfying
\eqref{consistency},
   \item[(ii)] $\{h_{\phi^n}(x_0)\}_{n=0}^{\infty}$ is a Stieltjes moment sequence which has an
S-representing measure $\rho$ supported in $[1,\infty)$,
   \item[(iii)] the sequence $\gamma_n=h_{\phi^n}(x_0)$, $n\in \Zp$, satisfies
\eqref{support-hankel}.
   \end{enumerate}
Moreover, each of these conditions implies that $C_\phi$ is
subnormal.
   \end{prop}
   No determinacy hypothesis is required here. To make the
construction behind the implication ``\text{(ii)$\implies$$C_\phi$ is
subnormal}'' explicit, put
   \begin{align*}
P_{x_r}(\varDelta) & =\frac{\mu(x_0)}{\mu(x_r)} \int_{\varDelta} t^r
\D \rho(t),\quad \varDelta \in\Borel(\mathbb{R}_+), \, 0\Le
r\Le\kappa,
   \\
P_{x_{1,j}}(\varDelta)&=\frac{\mu(x_0)}{\mu(x_{1,j})} \int_\varDelta
t^{j-1}(t^p-1) \D \rho(t),\quad \varDelta \in\Borel(\mathbb{R}_+), \,
j\in\N,
   \end{align*}
for $\sigma\in\Borel(\Rplus)$, where $p=\kappa+1$.
Positivity follows from $\supp(\rho)\subseteq[1,\infty)$, and
\eqref{cycle-id}--\eqref{circuit-diff} show that these measures
have mass one. Substitution in \eqref{consistency} verifies that
they satisfy the consistency condition. This is the construction
in \cite[Theorem~4.1.1 and proof of Proposition~4.3.4]{BJJS17}.

We also recall the realization procedure from
\cite[Example~4.3.5]{BJJS17}. For $\eta=1$, any sequence of positive
numbers $\gamma_n$ satisfying $\gamma_0=1$ and
$\gamma_{n+p}-\gamma_n>0$ for all $n\in\Zp$ can be realized as
$h_{\phi^n}(x_0)$ by setting
   \begin{equation}\label{realization}
   \left.
   \begin{gathered} \mu(x_0)=1,\qquad
\mu(x_r)=\gamma_r\quad(1\Le r\Le\kappa),
   \\
\mu(x_{1,n+1})=\gamma_{n+p}-\gamma_n\quad(n\in\Zp).
   \end{gathered}
   \;\; \right\}
   \end{equation}
The equality $h_{\phi^n}(x_0)=\gamma_n$ for all $n\in \Zp$ follows
from \eqref{cycle-id} and \eqref{circuit-diff} by induction in each
residue class modulo $p$. This procedure alone does not assert
subnormality or the Stieltjes moment property at the other vertices.
   \subsection{Affine transformations of Krein and Friedrichs measures}
\label{subsec:extremal}
   The existence of S-indeterminate Stieltjes moment sequences goes
back to Stieltjes \cite{Stieltjes}. Explicit examples relevant to the
present setting are also discussed in \cite{Simon98} (see also
\cite[Section~2]{BJJS17}). An H-representing measure $\nu$ of an
H-indeterminate Hamburger moment sequence
$\gammab=\{\gamma_n\}_{n=0}^{\infty}$ is called {\em N-extremal} if
the polynomials are dense in $L^2(\nu)$. The support of every
N-extremal measure is countably infinite and has no accumulation
point in $\mathbb{R}$ (see \cite[Theorems~5 and~4.11]{Simon98}).
Every S-indeterminate Stieltjes moment sequence has two distinguished
N-extremal S-representing measures, its {\em Krein measure}
$\widetilde\alpha$ and its {\em Friedrichs measure}
$\widetilde\beta$. The property used in our constructions is
   \begin{equation}\label{krein-gap}
0=\inf\supp(\widetilde\alpha)
<\inf\supp(\widetilde\beta).
   \end{equation}
See \cite{Simon98}; see also \cite[Subsection~2.1]{BJJS17}. In
particular, zero is an isolated atom of $\widetilde\alpha$. If the
zeroth moment is one, both measures are probability measures and
   \begin{align*}
0<\widetilde\alpha(\{0\})<1.
   \end{align*}

For $\vartheta>0$ and $a\in\mathbb{R}$, set
   \begin{equation}\label{affine-map}
\psi_{\vartheta,a}(t)=\vartheta(t+a),\quad t\in\mathbb{R},
   \end{equation}
as in \cite[(2.2.1)]{BJJS17}. The map $\psi_{\vartheta,a}$ is a
polynomial automorphism and $\psi_{\vartheta,a}^{-1} =
\psi_{\frac{1}{\vartheta},-a \vartheta}$. If $\nu$ is a Borel measure
on $\mathbb{R}$ having moments of all orders, the transported measure
$\nu\circ\psi_{\vartheta,a}^{-1}$ satisfies
   \begin{align*}
\int_{\mathbb{R}}t^n \D \, (\nu\circ\psi_{\vartheta,a}^{-1})(t)
=\vartheta^n\sum_{j=0}^{n}\binom{n}{j}a^{n-j} \int_{\mathbb{R}}t^j \D
\nu(t),\quad n\in\Zp.
   \end{align*}
Consequently, the transformation $\nu \longmapsto
\nu\circ\psi_{\vartheta,a}^{-1}$ sends measures with identical
moments to measures with identical moments. It preserves
N-extremality, and
   \begin{equation}\label{affine-support}
\supp(\nu\circ\psi_{\vartheta,a}^{-1})
=\psi_{\vartheta,a}(\supp(\nu)).
   \end{equation}
These are the assertions of \cite[Lemma~2.2.2]{BJJS17} needed in the
subsequent parts of this paper. The equality of supports follows from
the fact that $\psi_{\vartheta,a}$ is a homeomorphism. For
N-extremality, a change of variable gives a unitary map between the
two $L^2$-spaces
   \begin{align*}
L^2(\nu\circ\psi_{\vartheta,a}^{-1}) \ni f \longmapsto f\circ
\psi_{\vartheta,a} \in L^2(\nu),
   \end{align*}
which maps polynomials onto polynomials. If $a\Ge0$, a
measure supported in $\Rplus$ remains supported
in~$\Rplus$. We refer the reader to
\cite[Subsection~2.2]{BJJS17} for more complete
information on transforming moments via homotheties.
   \subsection{AI Disclosure} All ideas and contributions in this
manuscript originated exclusively with the authors.
They used OpenAI's GPT-5.6 tools to assist in
reviewing proof strategies and in literature searches.
   \section{The ``moment property implies subnormality''
problem for $C_{\phi_{1,\kappa}}$}
   \subsection{Negative solution for $\gcal_{1,\kappa}$
with $\kappa\Ge 1$} \label{sec:long}
   \begin{proof}[Proof of Theorem~\ref{cnyz}]
Fix \(\kappa\in\mathbb{N}\) and put $p=\kappa+1$. We divide the proof
into several steps.

\medskip

{\sc Step 1}. {\em Construction of N-extremal measures $\alpha_s$ and
$\beta_s$}.

Choose any S-indeterminate Stieltjes moment sequence, normalized so
that its zeroth term is \(1\). Let \(\widetilde{\alpha}\) and
\(\widetilde{\beta}\) be the Krein and Friedrichs measures of this
sequence, respectively. Clearly, $\widetilde{\alpha}$ and
$\widetilde{\beta}$ are probability measures. By \eqref{krein-gap},
   \begin{align*}
0=\inf\supp(\widetilde{\alpha}) < \inf\supp(\widetilde{\beta}).
   \end{align*}
Set
   \begin{align*}
b:=\inf\supp(\widetilde{\beta})>0.
   \end{align*}
Since the closed support of every N-extremal measure is countably
infinite and has no accumulation point in \(\mathbb{R}\) (see
\cite[Lemma~2.1.1]{BJJS17}), we deduce that \(0\) is an isolated
point of \(\supp(\widetilde{\alpha})\). Consequently,
   \begin{align} \label{wxlm}
m:=\widetilde{\alpha}(\{0\}) \in (0,1) \;\; \text{and} \;\; d:=
\inf\bigl(\supp(\widetilde{\alpha})\setminus\{0\}\bigr) > 0.
   \end{align}
In particular,
\[
\supp(\widetilde{\alpha})\setminus\{0\}
\subseteq[d,\infty).
\]

Fix \(q\in(0,1)\). For \(s>0\), define
\[
\psi_s(t)=q+st = \psi_{s,q/s}(t), \quad t\in\mathbb{R},
\]
where the notation \(\psi_{s,q/s}\) is as in \eqref{affine-map}.
Put
\[
\alpha_s = \widetilde{\alpha}\circ\psi_s^{-1} \;\; \text{and} \;\;
\beta_s = \widetilde{\beta}\circ\psi_s^{-1}.
\]
It follows from \eqref{affine-support} that
   \begin{align} \label{cvzq}
\supp{\alpha_s} & = \psi_s(\supp(\widetilde{\alpha})) \subseteq
[q,\infty),
   \\  \label{smbs}
\supp{\beta_s} & = \psi_s(\supp(\widetilde{\beta})) \subseteq \,
[q+sb,\infty).
   \end{align}
Combining \eqref{wxlm} and \eqref{cvzq}, we deduce that
$\alpha_{s,+}:=\alpha_s-m\delta_q$ is a nonzero positive Borel
measure on $\mathbb{R}_+$ with moments of all orders such that
   \begin{align*}
\supp(\alpha_{s,+}) = \psi_s(\supp(\widetilde{\alpha}) \setminus
\{0\}) \subseteq[q+sd,\infty).
   \end{align*}
By Subsection~\ref{subsec:extremal}, \(\alpha_s\) and \(\beta_s\) are
distinct N-extremal measures representing the same Stieltjes moment
sequence, that is,
   \begin{equation} \label{rywm}
\int_{\mathbb{R}_+} t^n\,\D\alpha_s(t) = \int_{\mathbb{R}_+} t^n \D
\beta_s(t), \quad n\in\mathbb{Z}_+.
   \end{equation}

\medskip

{\sc Step 2}. {\em Construction of Stieltjes moment sequences
$\{\gamma_n^{(s)}\}_{n=0}^{\infty}$ and
$\{D_n^{(s)}\}_{n=0}^{\infty}$}.

Fix $s_0 > 0$ such that $q+s_0 b>1$. For
\(r\in\{0,\ldots,\kappa\}\),~set
\[
I_r(s) = \int_{\mathbb{R}_+} \frac{t^r}{t^p-1} \D \beta_s(t), \quad
s\Ge s_0.
\]
Since, by \eqref{smbs}, $\inf\supp(\beta_s)=q+sb$ and $r\Le p-1$,
these integrals are finite and
   \begin{align*}
0\Le I_r(s) \Le \sup_{t\Ge q+sb} \frac{t^r}{t^p-1} \longrightarrow 0
\quad \text{as } s \rightarrow \infty.
   \end{align*}
This implies that
   \begin{equation}\label{eqlt}
\lim_{s\to\infty} I_r(s) = 0, \quad r=0,\ldots,\kappa.
   \end{equation}

Put
   \begin{align*}
A:=\frac{m}{1-q^p} > 0 \quad\text{and}\quad R_s:=I_0(s) + A > 0,
\quad s\Ge s_0.
   \end{align*}
For $s\Ge s_0$, define the Borel probability measure \(\rho_s\) on
\(\mathbb{R}_+\) by
   \begin{equation}\label{eqrs}
\rho_s(\varDelta) = \frac{1}{R_s} \left(
\int_\varDelta\frac{1}{t^p-1} \D \beta_s(t) + A\delta_q(\varDelta)
\right), \quad \varDelta\in\mathfrak{B}(\mathbb{R}_+).
   \end{equation}
The measure \(\rho_s\) has moments of all orders. For $s\Ge s_0$, set
   \begin{align} \label{gmsr}
\gamma_n^{(s)} = \int_{\mathbb{R}_+}t^n \D \rho_s(t), \quad
n\in\mathbb{Z}_+.
   \end{align}
Clearly, $\gamma_0^{(s)}=1$. For \(r\in\{0,\ldots,\kappa\}\), by
\eqref{eqrs}, we have
   \begin{align*}
\gamma_r^{(s)} = \frac{I_r(s)+Aq^r}{I_0(s)+A} > 0, \quad s\Ge s_0.
   \end{align*}
Therefore, by \eqref{eqlt},
   \begin{equation*}
\lim_{s\to \infty} \gamma_r^{(s)} = q^r, \quad r=0,\ldots,\kappa.
   \end{equation*}
Hence, by enlarging $s_0$ if necessary, we may also assume that
   \begin{equation} \label{gml1}
0<\gamma_1^{(s)}<1 \;\; \text{and} \;\; 0<\gamma_\kappa^{(s)}<1 \;\;
\text{for all} \;\; s \Ge s_0.
   \end{equation}
We now calculate the \(p\)-step differences of
\(\{\gamma_n^{(s)}\}_{n=0}^{\infty}\) for $s \Ge s_0$:
   \allowdisplaybreaks
   \begin{align} \nonumber
D_n^{(s)} :=\gamma_{n+p}^{(s)}-\gamma_n^{(s)} &
\overset{\eqref{eqrs}}{=} \frac{1}{R_s} \left( \int_{\mathbb{R}_+}t^n
\D \beta_s(t) - mq^n \right)
   \\ \nonumber
& \overset{\eqref{rywm}}{=} \frac{1}{R_s} \left(
\int_{\mathbb{R}_+}t^n \D \alpha_s(t) - mq^n \right)
   \\ \label{dyrs}
& \hspace{1ex}= \frac{1}{R_s} \int_{\mathbb{R}_+}t^n \D
\alpha_{s,+}(t), \qquad n\in\mathbb{Z}_+.
   \end{align}
Consequently, for every $s \Ge s_0$, the sequence
$\{D_n^{(s)}\}_{n=0}^{\infty}$ is a Stieltjes moment sequence of
positive real numbers.

Now fix $s \Ge s_0$. For the remainder of the proof, we suppress the
dependence on $s$ in the notation, writing $\alpha, \beta,
\alpha_{+}, \rho, D_n, R, \gamma_n$ for $\alpha_s, \beta_s,
\alpha_{s,+}, \rho_s, D_n^{(s)}, R_s, \gamma_n^{(s)}$, respectively.

\medskip

{\sc Step 3}. {\em Construction of the composition operator
$C_{\phi}$}.

Let $X=X_{1,\kappa}$ be as in \eqref{graph-set}, and let
$\phi=\phi_{1,\kappa}$ be the self-map defined by \eqref{graph-map}.
Define the discrete measure \(\mu\) on \(X\) by
   \begin{align} \label{myzr}
   \left.
   \begin{gathered}
\mu(x_0):=1 \;\; \text{and} \;\; \mu(x_r):=\gamma_r \;\; \text{for}
\;\; r=1,\ldots,\kappa,
   \\ 
\mu(x_{1,n+1}):=D_n \;\; \text{for} \;\; n\in\mathbb{Z}_+.
   \end{gathered}
   \;\; \right\}
   \end{align}
All these numbers are finite and strictly positive. This is precisely
the construction \eqref{realization} from
\cite[Example~4.3.5]{BJJS17}.

First observe that formula \eqref{cycle-id} gives
   \begin{align}  \label{grzq}
h_{\phi^r}(x_0) = \frac{\mu(x_r)}{\mu(x_0)} \overset{\eqref{myzr}}{=}
\gamma_r, \quad r=0,\ldots,\kappa.
   \end{align}
Moreover, by \eqref{hfin} and \eqref{myzr},
   \begin{equation} \label{hfxz}
h_{\phi^n}(x_{1,1}) = \frac{\mu(x_{1,n+1})}{\mu(x_{1,1})} =
\frac{D_n}{D_0}, \quad n\in\mathbb{Z}_+.
   \end{equation}
In turn, \eqref{circuit-diff}, \eqref{hfxz} and \eqref{dyrs} yield
   \begin{align*}
h_{\phi^{n+p}}(x_0)-h_{\phi^n}(x_0) = \frac{\mu(x_{1,1})}{\mu(x_0)}
h_{\phi^n}(x_{1,1}) = D_0\frac{D_n}{D_0} = \gamma_{n+p}-\gamma_n,
\quad n \in \mathbb{Z}_+.
   \end{align*}
An induction applied separately to each residue class modulo $p$,
together with \eqref{grzq}, shows that
   \begin{equation}\label{hvgz}
h_{\phi^n}(x_0)=\gamma_n, \quad n\in\mathbb{Z}_+.
   \end{equation}

\medskip

{\sc Step 4}. {\em Generation of Stieltjes moment sequences}.

We first construct S-representing measures for all sequences of
iterated Radon-Nikodym derivatives. Combining \eqref{cycle-id} with
\eqref{myzr} and \eqref{hvgz}, we obtain
   \begin{equation*}
h_{\phi^n}(x_r) = \frac{\gamma_{n+r}}{\gamma_r}, \quad
n\in\mathbb{Z}_+, \, r = 0,\ldots,\kappa.
   \end{equation*}
Hence, by \eqref{gmsr}, the Borel probability measure $P_r$ on
$\mathbb{R}_+$ defined by
   \begin{align*}
P_r(\varDelta) = \frac{1}{\gamma_r} \int_\varDelta t^r \D \rho(t),
\quad \varDelta\in\Borel(\mathbb{R}_+),
   \end{align*}
is an S-representing measure of
\(\{h_{\phi^n}(x_r)\}_{n=0}^{\infty}\) for every $r=0,\ldots,\kappa$.

For \(j\in\mathbb{N}\), formula \eqref{hfin} yields
   \begin{equation} \label{hinm1}
h_{\phi^n}(x_{1,j}) = \frac{\mu(x_{1,n+j})}{\mu(x_{1,j})}
\overset{\eqref{myzr}}{=} \frac{D_{n+j-1}}{D_{j-1}}, \quad
n\in\mathbb{Z}_+.
   \end{equation}
By \eqref{dyrs}, the Borel probability measure \(Q_j\) given by
   \begin{align*}
Q_j(\varDelta) = \frac{1}{D_{j-1}R} \int_\varDelta t^{j-1} \D
\alpha_+(t), \quad \varDelta\in\Borel(\mathbb{R}_+),
   \end{align*}
satisfies
   \begin{align*}
\int_{\mathbb{R}_+}t^n \D Q_j(t) = \frac{D_{n+j-1}}{D_{j-1}}
\overset{\eqref{hinm1}}{=} h_{\phi^n}(x_{1,j}), \quad
n\in\mathbb{Z}_+.
   \end{align*}
Hence, we have proved that for every $x\in X$,
$\{h_{\phi^n}(x)\}_{n=0}^{\infty}$ is a Stieltjes moment sequence.
Using Lemma~\ref{wkws}, we obtain (i).

\medskip

{\sc Step 5}. {\em Failure of hyponormality of $C_{\phi}$}.

We apply the hyponormality criterion from Proposition~\ref{wkwh} at
the vertex \(x=x_\kappa\). Since $\kappa\Ge1$, we have
   \begin{align*}
\phi^{-1}(\{x_\kappa\}) = \{x_0,x_{1,1}\}, \;\;
\phi^{-1}(\{x_0\})=\{x_1\} \;\; \text{and} \;\;
\phi^{-1}(\{x_{1,1}\})=\{x_{1,2}\}.
   \end{align*}
Consequently, the left-hand side of inequality~\eqref{hypcriterion}
at \(x=x_\kappa\) satisfies
   \allowdisplaybreaks
   \begin{align*}
\frac{1}{\mu(x_\kappa)} \sum_{y\in\phi^{-1}(\{x_\kappa\})}
\frac{\mu(y)^2}{\mu\bigl(\phi^{-1}(\{y\})\bigr)} =
\frac{1}{\gamma_\kappa} \left( \frac{1}{\gamma_1} + \frac{D_0^2}{D_1}
\right) > \frac{1}{\gamma_\kappa\gamma_1} \overset{\eqref{gml1}}{>}1.
   \end{align*}
Thus, inequality~\eqref{hypcriterion} fails at \(x=x_\kappa\), so by
Proposition~\ref{wkwh}, \(C_\phi\) is not hyponormal. This proves
(ii) and completes the proof.
   \end{proof}
   \subsection{Positive solution for $\gcal_{1,0}$} \label{sec:loop}
   \begin{proof}[Proof of Theorem~\ref{kezo}]
\textup{(i)}$\Rightarrow$\textup{(ii)}. Since $\phi$ has finite
fibers, Lemma~\ref{wkws} implies that $\Dinf(C_\phi)$ is dense in
$L^2(\mu)$. By \cite[Proposition~3.2.1]{BJJS12}, $C_\phi$ generates
Stieltjes moment sequences.

\textup{(ii)}$\Rightarrow$\textup{(iii)}. This implication follows
from Lemma~\ref{wkws}.

\textup{(iii)}$\Rightarrow$\textup{(i)}. It follows from
\eqref{circuit-diff} that
\[
h_{\phi^{n+1}}(x_0) = h_{\phi^n}(x_0) + \frac{\mu(x_{1,1})}{\mu(x_0)}
h_{\phi^n}(x_{1,1}), \qquad n\in\Zp.
\]
Since $\{h_{\phi^n}(x_{1,1})\}_{n=0}^{\infty}$ is a Hamburger moment
sequence, we deduce that (cf.\ \cite[Theorem~3.8]{Sch17})
   \begin{align*}
\sum_{i,j=0}^{N} \bigl(h_{\phi^{i+j+1}}(x_0) -
h_{\phi^{i+j}}(x_0)\bigr) \lambda_i\overline{\lambda_j} =
\frac{\mu(x_{1,1})}{\mu(x_0)} \sum_{i,j=0}^{N}
h_{\phi^{i+j}}(x_{1,1})\lambda_i\overline{\lambda_j} \Ge 0,
   \end{align*}
for all \(\lambda_0,\ldots,\lambda_N\in\mathbb{C}\). This is
precisely condition~\textup{(iii)} in Proposition~\ref{dirze}. By
this proposition, \(C_\phi\) is subnormal.
   \end{proof}
   \subsection{Partial positive solutions for $\gcal_{1,\kappa}$
with $\kappa \Ge 1$} \label{sec:det}
   Theorem~\ref{cnyz} shows that Theorem~\ref{kezo} is specific to
the case \(\kappa=0\). For every \(\kappa\Ge1\), the assumption
\[
\{h_{\phi^n}(x_0)\}_{n=0}^{\infty} \quad\text{and}\quad
\{h_{\phi^n}(x_{1,1})\}_{n=0}^{\infty} \text{ are Stieltjes moment
sequences}
\]
do not even imply the hyponormality of \(C_\phi\). Additional
assumptions, such as the S-determinacy or the Carleman condition
appearing in Proposition~\ref{cpws} below, are therefore essential.
   \begin{prop} \label{cpws}
Let $\kappa\Ge 1$, let $\mu$ be a discrete measure on $X_{1,\kappa}$,
and let $C_\phi$ be the composition operator in $L^2(\mu)$ with
$\phi=\phi_{1,\kappa}$. Suppose that
$\{h_{\phi^n}(x)\}_{n=0}^{\infty}$ is a Stieltjes moment sequence for
$x\in \{x_0, x_{1,1}\}$, and at least one of the following two
conditions is satisfied{\em :}
   \begin{enumerate}
\item[(i)] the Stieltjes moment sequence
$\{\gamma_{jp}\}_{j=0}^{\infty}$ is S-determinate,
   \item[(ii)] the sequence $\{\gamma_{jp}\}_{j=0}^{\infty}$ satisfies the Carleman
condition,
   \item[(iii)] the sequence $\{\gamma_{(j+1)p}-\gamma_{jp}\}_{j=0}^{\infty}$
satisfies the Carleman condition,
   \end{enumerate}
where $p:=\kappa+1$ and $\gamma_n:=h_{\phi^n}(x_0)$ for all
$n\in\Zp$. Then $C_\phi$ is subnormal.
   \end{prop}
   \begin{proof}
By \eqref{circuit-diff},
   \begin{align*}
\gamma_{n+p}-\gamma_n = \frac{\mu(x_{1,1})}{\mu(x_0)}
h_{\phi^n}(x_{1,1}), \qquad n\in\mathbb{Z}_+.
   \end{align*}
Consequently, by assumption,
\[
\{\gamma_{n+p}-\gamma_n\}_{n=0}^{\infty}
\]
is a Stieltjes moment sequence.

(i) Suppose that $\{\gamma_{jp}\}_{j=0}^{\infty}$ is S-determinate.
Since \(\{\gamma_{n+p}-\gamma_n\}_{n=0}^{\infty}\) is a Stieltjes
moment sequence, its subsequence
\[
\{\gamma_{(j+1)p}-\gamma_{jp}\}_{j=0}^{\infty}
\]
is also a Stieltjes moment sequence. By
\cite[Lemma~4.3.1(ii)]{BJJS17}, recalled in
Subsection~\ref{subsec:moments}, the Stieltjes moment sequence
\(\{\gamma_n\}_{n=0}^{\infty}\) is S-determinate and that its unique
S-representing measure vanishes on \([0,1)\). Proposition~\ref{dirze}
therefore implies that \(C_\phi\) is subnormal.

(ii) This is a direct consequence of (i) and the fact that the
Carleman condition implies S-determinacy.

Finally, observe that conditions (ii) and (iii) are equivalent by
applying \eqref{carleman-diff} to the sampled sequence
$\{\gamma_{jp}\}_{j=0}^{\infty}$.
   \end{proof}
   \section{A negative solution to the injectivity problem} \label{sec:inj}
   \begin{proof}[Proof of Theorem~\ref{noninjective}]
We divide the proof into several steps.

\medskip

{\sc Step 1}. {\em Selection of two representing measures}.

Choose any S-indeterminate Stieltjes moment sequence with normalized
zeroth term, and denote by \(\widetilde{\alpha}\) and
\(\widetilde{\beta}\) the Krein and Friedrichs measures of this
sequence, respectively (see Step~1 of the proof of
Theorem~\ref{cnyz}). Then
$\widetilde{\alpha}(\mathbb{R}_+)=\widetilde{\beta}(\mathbb{R}_+)=1$.
According to \eqref{krein-gap},
   \begin{align*}
0=\inf\supp(\widetilde{\alpha}) < \inf\supp(\widetilde{\beta}) \;\;
\text{and} \;\; b:=\inf\supp(\widetilde{\beta})>0.
   \end{align*}
Choose a real number $\vartheta>\frac{1}{b}$ and let
$\psi_{\vartheta,0}$ be given by (see \eqref{affine-map})
\[
\psi_{\vartheta,0}(t)=\vartheta t, \qquad t\in\mathbb{R}.
\]
Define the Borel probability measures \(\alpha\) and \(\beta\) on
\(\mathbb{R}\) by
   \begin{align*}
\alpha = \widetilde{\alpha} \circ \psi_{\vartheta,0}^{-1} \;\;
\text{and} \;\; \beta = \widetilde{\beta} \circ
\psi_{\vartheta,0}^{-1}.
   \end{align*}
It follows from Subsection~\ref{subsec:extremal} that \(\alpha\) and
\(\beta\) are N-extremal measures of the same Stieltjes moment
sequence, and
   \begin{align} \label{inj-cvzq}
& \supp{\alpha} = \psi_{\vartheta,0}(\supp(\widetilde{\alpha}))
\subseteq \Rplus \;\; \text{and} \;\; \inf\supp(\alpha)=0,
   \\  \label{inj-smbs}
&\supp{\beta} = \psi_{\vartheta,0}(\supp(\widetilde{\beta}))
\subseteq \, [\vartheta b,\infty) \subseteq \Rplus \;\; \text{and}
\;\; \inf\supp(\beta)=\vartheta b>1.
   \end{align}
Set
   \begin{align}  \label{gncx}
\gamma_n :=\int_{\Rplus}t^n \D \alpha(t), \quad n\in\Zp.
   \end{align}
Note that $\gamma_0=1$. By Subsection~\ref{subsec:extremal}, the
support of \(\alpha\) has no accumulation point in \(\mathbb{R}\).
Hence, in view of \eqref{inj-cvzq}, the point $0$ is an isolated
point of $\supp(\alpha)$, so
   \begin{align*}
\alpha(\{0\}) \in (0,1).
   \end{align*}
Set
   \begin{align}  \label{djm1}
d_j:=\gamma_j-\gamma_{j-1}, \quad j\in\N.
   \end{align}
Because, by \eqref{inj-smbs}, $\beta$ is a probability measure
supported in \((1,\infty)\), we have
   \begin{equation}
\label{djbe} d_j = \int_{\Rplus}t^{j-1}(t-1)\,\D \beta(t)>0, \qquad
j\in\N.
   \end{equation}

\medskip

{\sc Step 2}. {\em Construction of the composition operator
$C_{\phi}$}.

Let $X=\{z,x_0,x_1,x_2,\ldots\}$, where all points are chosen to be
distinct. Define a self-map \(\phi\colon X\to X\) by
   \begin{align} \label{dzwc}
   \left.
   \begin{gathered} \phi(x_0)=x_0, \quad \phi(z)=x_0, \quad
\phi(x_1)=x_0,
   \\
\phi(x_{j+1})=x_j, \quad j\in\N.
   \end{gathered}
   \;\: \right\}
   \end{align}

Thus, $z\notin\phi(X)$. The directed graph $\gcal$ induced by $\phi$
is shown in Figure~\ref{fig:leaf}.
   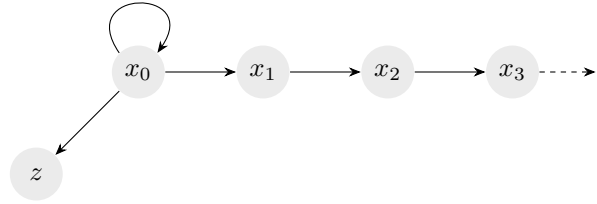
\begin{figure}[htbp]
   \centering
\begin{tikzpicture}[x=1cm,y=1cm]
\node[ain vertex] (v0) at (0,0) {$x_0$}; \node[ain vertex] (v1) at
(1.65,0) {$x_1$}; \node[ain vertex] (v2) at (3.30,0) {$x_2$};
\node[ain vertex] (v3) at (4.95,0) {$x_3$}; \node[ain vertex] (vz) at
(-1.35,-1.35) {$z$}; \draw[ain edge] (v0.north west) .. controls
(-.85,1.12) and (.85,1.12) .. (v0.north east); \draw[ain edge] (v0)
-- (v1); \draw[ain edge] (v1) -- (v2); \draw[ain edge] (v2) -- (v3);
\draw[ain continuation] (v3) -- (6.05,0); \draw[ain edge] (v0) --
(vz);
\end{tikzpicture}
\caption{\label{fig:leaf}The directed graph $\gcal$: one loop, one
infinite branch, and one leaf $z$. The arrows $(\phi(y),y)$ are
directed opposite to the action~of~$\phi$.}
   \end{figure}

   Fix $r\in \bigl(0,\alpha(\{0\})\bigr)$. Define a discrete measure
\(\mu\) on \(2^X\) by
   \begin{equation} \label{jmzd}
\mu(x_0)=1, \quad \mu(z)=r, \quad \mu(x_j)=d_j \quad (j\in\N).
   \end{equation}
Hence, the composition operator $C_{\phi}$ is well defined.

\medskip

{\sc Step 3}. {\em Calculation of the Radon-Nikodym derivatives}.

Formulas \eqref{dzwc} give
   \begin{equation}\label{przn}
   \left.
   \begin{aligned}
\phi^{-n}(\{x_0\}) &= \{z,x_0,x_1,\ldots,x_n\}, && n\in\N,
   \\
\phi^{-n}(\{x_j\}) &= \{x_{j+n}\}, && n\in\Zp,\ j\in\N,
   \\
\phi^{-n}(\{z\}) &= \varnothing, && n\in\N.
   \end{aligned}
   \;\; \right\}
   \end{equation}
In turn, formulas \eqref{hfin}, \eqref{jmzd}, and \eqref{przn} yield
   \begin{align}\label{hfpr}
h_{\phi^n}(x_0) =1+r+\sum_{j=1}^{n}d_j
\overset{\eqref{djm1}}=1+r+\gamma_n-\gamma_0=\gamma_n+r, \quad
n\in\N.
   \end{align}
Similarly, \eqref{przn} gives
   \begin{equation*}
h_{\phi^n}(x_j)=\frac{d_{j+n}}{d_j},
\quad j\in\N,\ n\in\Zp,
   \end{equation*}
whereas \eqref{przn} yields
   \begin{equation}\label{hzzo}
h_{\phi^n}(z)=0,\quad n\in\N.
   \end{equation}
Of course, $h_{\phi^0}(x)=1$ for every $x\in X$.

\medskip

{\sc Step 4}. {\em $C_\phi$ generates Stieltjes moment sequences}.

Define Borel probability measures $\{P_x\}_{x\in X}$ on $\Rplus$ by
   \begin{equation} \label{pzoo}
   \left.
   \begin{aligned}
P_z&=\delta_0,
   \\
P_{x_0}&=\alpha-r\delta_0+r\delta_1,
   \\
P_{x_j}(\varDelta) &=\frac{1}{d_j}\int_{\varDelta}t^{j-1}(t-1)
\D\beta(t), \quad \varDelta\in\Borel(\Rplus),\ j\in\N.
   \end{aligned}
   \;\; \right\}
   \end{equation}
By \eqref{inj-cvzq}, the measure $P_{x_0}$ is positive because
$0<r<\alpha(\{0\})$, and its total mass is $\alpha(\Rplus)-r+r=1$.
For each $j\in \N$, the measure $P_{x_j}$ is positive because, by
\eqref{inj-smbs}, $\beta$ is supported in $(1,\infty)$, while
\eqref{djbe} shows that its total mass is one.

Formulas \eqref{gncx}, and \eqref{pzoo} imply that
   \begin{equation}\label{hfmz}
\int_{\Rplus}t^n \D P_{x_0}(t) =\gamma_n-r\, 0^n+r\,
1^n=\gamma_n+r\overset{\eqref{hfpr}}=h_{\phi^n}(x_0), \quad n\in\N.
   \end{equation}
Clearly, both sides of \eqref{hfmz} are also equal for $n=0$. Next,
observe that
   \begin{align*}
\int_{\Rplus}t^n \D P_{x_j}(t) &
\overset{\eqref{pzoo}}=\frac{1}{d_j}\int_{\Rplus}t^{n+j-1}(t-1)\,\D
\beta(t)
   \\
&\hspace{.5ex}\overset{\eqref{gncx}}=\frac{\gamma_{n+j}-\gamma_{n+j-1}}{d_j}
   \\
& \hspace{.5ex}\overset{\eqref{djm1}}=\frac{d_{n+j}}{d_j}
   \\
&\overset{\eqref{hfpr}} =h_{\phi^n}(x_j), \qquad n\in\Zp, \, j\in\N.
   \end{align*}
Finally, \eqref{hzzo} and \eqref{pzoo} give $\int_{\Rplus}t^n \D
P_z(t)=h_{\phi^n}(z)$ for all $n\in\Zp$. Thus
   \begin{equation*}
h_{\phi^n}(x)=\int_{\Rplus}t^n \D P_x(t), \quad n\in\Zp, \, x\in X,
   \end{equation*}
so $\{h_{\phi^n}(x)\}_{n=0}^{\infty}$ is a Stieltjes moment sequence
for every $x\in X$. Since, by \eqref{przn}, the self-map $\phi$ has
finite fibers, Lemma~\ref{wkws} implies that $C_{\phi}$ generates
Stieltjes moment sequences.

\medskip

{\sc Step 5}. {\em Failure of injectivity}.

Since $z\notin\phi(X)$, we have
   \[
C_\phi\chi_{\{z\}}=\chi_{\{z\}}\circ\phi=0 \;\; \text{and} \;\;
\|\chi_{\{z\}}\|^2=\mu(z)=r>0.
   \]
Hence $\chi_{\{z\}}$ is a nonzero element of $\ker C_\phi$, so
$C_\phi$ is not injective. This completes the proof.
   \end{proof}
   \begin{remark}  \label{unbco}
The operator $C_\phi$ constructed in the proof of
Theorem~\ref{noninjective} is necessarily unbounded. Indeed, if it
were bounded, then, since it generates Stieltjes moment sequences, it
would be subnormal by Lambert's characterization (see \cite{Lam76};
see also \cite[Theorem~7]{SS89} for the version without the
injectivity assumption). It would then be hyponormal and hence
injective (see \cite[Theorem~9d]{Ha-Wh84}), contrary to its
construction. Moreover, by \eqref{hzzo}, $h_\phi(z)=0$, so
\eqref{suxj} implies that $\phi$ is not surjective. Thus $C_\phi$ is
not hyponormal by Proposition~\ref{wkwh}.
   \end{remark}
   \begin{remark} \label{gstr}
The directed graph $\gcal$ in the proof of Theorem~\ref{noninjective}
is obtained from $\gcal_{1,0}$ by attaching a single leaf $z$ to the
loop vertex $x_0$ (see Figures~\ref{fig:onebranch}
and~\ref{fig:leaf}). More precisely, besides the loop $(x_0,x_0)$ and
the infinite branch $x_1,x_2,\ldots$, it has the additional edge
$(x_0,z)$. Since $\phi(X)=X\setminus\{z\}$, it follows from
\cite[Lemma~11]{BJJS18} and \eqref{hfi} that the kernel of the
composition operator $C_\phi$ is given explicitly by
   \[
\ker C_\phi=\mathbb{C}\cdot\chi_{\{z\}}.
   \]

If the leaf $z$ is removed while all the remaining point masses of
$\mu$ are left unchanged, where $\mu$ is as in \eqref{jmzd}, the
resulting composition operator is over $\gcal_{1,0}$ with symbol
$\phi_{1,0}$. Using \eqref{gncx} and \eqref{hfpr}, we deduce that the
iterated Radon-Nikodym derivatives
$\{h_{\phi_{1,0}^n}(x_0)\}_{n=0}^{\infty}$, calculated with respect
to the restriction of $\mu$ to $X_{1,0}$, form the Stieltjes moment
sequence $\{\gamma_n\}_{n=0}^{\infty}$ represented by $\alpha$,
whereas, for every $j\in\N$, the sequence
$\{h_{\phi_{1,0}^n}(x_j)\}_{n=0}^{\infty}$ is a Stieltjes moment
sequence represented by $P_{x_j}$ (see \eqref{pzoo}). Hence, by
Theorem~\ref{kezo}, the composition operator $C_{\phi_{1,0}}$ is
subnormal.

The role of the leaf can now be seen directly. Attaching $z$ with
mass $\mu(z)=r$ changes the sequence of iterated Radon-Nikodym
derivatives at $x_0$ from $\{\gamma_n\}_{n=0}^{\infty}$~to
   \[
1,\ \gamma_1+r,\ \gamma_2+r,\ldots.
   \]
The measure $P_{x_0}$ in \eqref{pzoo} represents precisely this new
sequence: it is obtained from $\alpha$ by partially transferring the
mass concentrated at $0$ to the point $1$. Thus, the atom of the
Krein measure at zero makes it possible to attach the leaf, thereby
destroying injectivity while preserving the Stieltjes moment
property.
   \end{remark}


\begin{thebibliography}{99}

\bibitem{AlCar65} W.~A. Al-Salam, L. Carlitz, \emph{Some orthogonal $q$-polynomials}, Math.
Nachr. {\bf 30} (1965), 47--61.

\bibitem{B25}
C.~Berg, \emph{On the entropy for indeterminate moment problems},
preprint, arXiv:2511.12684v2 (2025).
https://arxiv.org/abs/2511.12684v2.

\bibitem{BCR84}
C.~Berg, J.~P.~R. Christensen, and P.~Ressel,
\emph{Harmonic Analysis on Semigroups}, Springer, Berlin, 1984.

\bibitem{BS26} C. Berg, and R. Szwarc, \emph{Special N-extremal solutions to
indeterminate moment problems}, Complex Anal. Oper. Theory
\textbf{20} (2026), Article 138. DOI: 10.1007/s11785-026-02000-9.

\bibitem{BV94}
C.~Berg, and G.~Valent, \emph{The Nevanlinna parametrization for some
indeterminate Stieltjes moment problems associated with birth and
death processes}, Methods Appl. Anal. \textbf{1} (1994), 169--209.

\bibitem{BJJS12} P.~Budzy\'nski, Z.~J. Jab{\l}o\'nski, I.~B. Jung,
and J.~Stochel, \emph{Unbounded subnormal weighted shifts on directed
trees}, J. Math. Anal. Appl. \textbf{394} (2012) 819--834.

\bibitem{BJJS14}
P.~Budzy\'nski, Z.~J. Jab{\l}o\'nski, I.~B. Jung, and J.~Stochel,
\emph{On unbounded composition operators in $L^2$-spaces},
Ann. Mat. Pura Appl. (4) \textbf{193} (2014), 663--688.

\bibitem{BJJS15}
P.~Budzy\'nski, Z.~J. Jab{\l}o\'nski, I.~B. Jung, and J.~Stochel,
\emph{Unbounded subnormal composition operators in $L^2$-spaces},
J. Funct. Anal. \textbf{269} (2015), 2110--2164.

\bibitem{BJJS17}
P.~Budzy\'nski, Z.~J. Jab{\l}o\'nski, I.~B. Jung, and J.~Stochel,
\emph{Subnormality of unbounded composition operators over
one-circuit directed graphs: Exotic examples},
Adv. Math. \textbf{310} (2017), 484--556.

\bibitem{BJJS18}
P.~Budzy\'nski, Z.~J. Jab{\l}o\'nski, I.~B. Jung, and J.~Stochel,
\emph{Unbounded Weighted Composition Operators in $L^2$-Spaces},
Lecture Notes in Mathematics, vol.~2209, Springer, Cham, 2018.

\bibitem{Chi68} T.~S. Chihara, \emph{On determinate Hamburger moment problems}, Pacific J.
Math. {\bf 27} (1968), 475--484.

\bibitem{Chr03} J.~S. Christiansen, \emph{The moment problem
associated with the Stieltjes--Wigert polynomials}, J. Math. Anal.
Appl. \textbf{277} (2003), 218--245.

\bibitem{Con91} J.~B. Conway, \emph{The theory of
subnormal operators}, Mathematical Surveys and Monographs, {\bf 36},
American Mathematical Society, Providence, RI, 1991.

\bibitem{Ha-Wh84} D. Harrington, and R. Whitley, \emph{Seminormal composition operators}, J.
Operator Theory \textbf{11} (1984), 125--135.

\bibitem{IsMa94} M.~E.~H. Ismail, and D.~R. Masson, \emph{$q$-Hermite polynomials, biorthogonal
rational functions, and $q$-beta integrals}, Trans. Amer. Math. Soc.
{\bf 346} (1994), 63--116.

\bibitem{JJS12}
Z.~J. Jab{\l}o\'nski, I.~B. Jung, and J.~Stochel,
\emph{A non-hyponormal operator generating Stieltjes moment
sequences}, J. Funct. Anal. \textbf{262} (2012), 3946--3980.

\bibitem{Lam76}
A.~Lambert, \emph{Subnormality and weighted shifts},
J. London Math. Soc. (2) \textbf{14} (1976), 476--480.

\bibitem{Pe95} H.~L. Pedersen, \emph{Stieltjes moment problems and the Friedrichs extension
of a positive definite operator}, J. Approx. Theory 83 (1995)
289--307.

\bibitem{Sch17} K. Schm\"{u}dgen, \emph{The moment
problem}, Graduate Texts in Mathematics 277, Springer, 2017.

\bibitem{Simon98}
B.~Simon, \emph{The classical moment problem as a self-adjoint
finite difference operator}, Adv. Math. \textbf{137} (1998), 82--203.

\bibitem{SS89}
J.~Stochel, and F.~H. Szafraniec, \emph{On normal extensions of
unbounded operators. II}, Acta Sci. Math. (Szeged) \textbf{53}
(1989), 153--177.

\bibitem{Stieltjes}
T.-J.~Stieltjes, \emph{Recherches sur les fractions continues},
Ann. Fac. Sci. Toulouse (1) \textbf{8} (1894), J1--J122;
\emph{Recherches sur les fractions continues [Suite et fin]},
\textbf{9} (1895), A5--A47.

\end{thebibliography}
\end{document}